\documentclass[11pt]{extarticle}
\usepackage[a4paper, total={6in, 8in}]{geometry}
\usepackage{amsthm}
\usepackage{graphicx}
\usepackage{amsfonts}
\usepackage{amsmath}

\newcommand{\mb}{\mathbb}
\newcommand{\mc}{\mathcal}

\newcommand{\Z}{\mb Z}
\newcommand{\bs}{\backslash}

\theoremstyle{plain}
\newtheorem{theorem}{Theorem}[section]
\newtheorem{lemma}[theorem]{Lemma}
\newtheorem{corollary}[theorem]{Corollary}
\newtheorem{proposition}[theorem]{Proposition}

 \theoremstyle{definition}
\newtheorem{definition}[theorem]{Definition}

\newtheorem{?}[theorem]{Problem}

\usepackage{amsmath,amssymb,tikz-cd}

\begin{document}

\title{On Functional Equations for the Elliptic Dilogarithm}

\author{Vasily Bolbachan}

\maketitle

\begin{abstract}
Let $E$ be an elliptic curve over an algebraically closed field of characteristic 0. We prove that the pre-Bloch group of the function field of $E$ can be generated by the functions of
degree not higher than 3. We apply this result to the elliptic dilogarithm function defined by S. Bloch. He has shown that any element of the pre-Bloch group gives a (so-called elliptic Bloch)
relation between the values of the so-called Elliptic dilogarithm. We conclude that any
elliptic Bloch relation can be reduced to the antisymmetry relation and the elliptic Bloch
relations for the functions of degree 3.

\end{abstract}

\section{Introduction}
\label{intro}
\subsection{Summary}

We recall that the Bloch-Wigner dilogarithm \cite{Bl1} is defined
 by the following formula:
$$D(z)=\Im Li_2(z)+\arg(1-z)\ln |z|.$$

This formula defines a real analytic function on $\mathbb C\backslash\{0,1\}$ which satisfies to the relation $D(z)=-D(z^{-1})$. 

Let $E=\mb C/\left<1,\tau\right>$ be an elliptic curve over $\mb C$. The  elliptic dilogarithm was defined by Spencer Bloch (\cite{Bl1}, see also \cite{zagier2000classical}). The equivalent representation is given by the following formula: 
$$D_\tau(\xi)=\sum\limits_{n=-\infty}^{\infty}D(e^{2\pi i \xi+2\pi i\tau n}).$$
This formula defines a real analytic function on $E$. From the relation $D(z)=-D(z^{-1})$ it is easy to deduce the following antisymmety relation:
\begin{equation}
\label{rel1}
D_\tau(\xi)+D_\tau(-\xi)=0.
    \end{equation}

Starting from here we assume that $E$ is defined over an arbitrary algebraically closed field $k$ of characteristic zero.
    
Denote by $\mathbb Z[E]$ a free abelian group generated by the points of $E$. For a point $z\in E$ we denote by $[z]$ the corresponding element in the group $\mathbb Z[E]$. When $k=\mathbb C$, the elliptic dilogarithm gives a well-defined map $\widetilde D_\tau\colon \mathbb Z[E]\to \mathbb C$ defined by the formula $\widetilde D_\tau([z])=D_\tau(z)$.

 For a rational function $g$ on some smooth algebraic curve, denote by $(g)$ its divisor. Let us formulate the so-called elliptic Bloch relations (\cite[Theorem 9.2.1]{Bl1}, see also \cite{zagier2000classical}). Let $f$ be a rational function on $E$ of degree $n$ such that
$$(f)=\sum\limits_{i=1}^{n}([\alpha_i]-[\gamma_i]),(1-f)=\sum\limits_{i=1}^{n}([\beta_i]-[\gamma_i]).$$
 Define the element $\eta_f\in \mathbb Z[E]$ by the following formula:
\begin{equation}
\label{rel2}
\eta_f=\sum\limits_{i,j=1}^{n}\left([\alpha_i-\beta_j]+[\beta_i-\gamma_j]+[\gamma_i-\alpha_j]\right).
    \end{equation}
The following definition is taken from \cite{GL}:

\begin{definition}
Define a subgroup $\mathcal R(E)$  of the group $\mathbb Z[E]$ generated by the following elements:
\begin{enumerate}
    \item $\eta_f$, where $f\in k(E)$,
    \item $[z]+[-z]$, where $z\in E$,
    \item  $2\cdot(z-\sum\limits_{2z'=z}[z'])$, where $z\in E$.
\end{enumerate}
\emph{The elliptic Bloch group} $B_3(E)$ is defined as the quotient $\mathbb Z[E]/\mathcal R(E)$.
\end{definition}

According to \cite{GL}, when $k=\mathbb C$, the map $\widetilde D_\tau$ is zero on $\mathcal R(E)$.

Consider the following geometric example (see \cite[Lemma 3.29]{GL}). Let us realize the elliptic curve $E$ as a cubic plane curve in $\mb P^2$ and consider three different lines $l,m,n\subset \mb P^2(\mb C)$ intersecting at a point in the complement of $E$. Let $h_l,h_m,h_n$ be homogeneous equations of these lines such that $h_m=h_n+h_l$. Denote by $\alpha_i,\beta_i,\gamma_i, i=1,2,3$ the intersection points of the lines $l,n,m$ with $E$. Denote the element $\eta_f$, where $f=h_l/h_m$, by $\eta_{l,m,n}$. This element has the following form:
\begin{equation}
\label{rel3}
\eta_{l,m,n}=\sum\limits_{i,j=1}^3\left([\alpha_i-\beta_j]+[\beta_i-\gamma_j]+[\gamma_i-\alpha_j]\right).
    \end{equation}

Our main result is the following statement:

\begin{theorem}
\label{thm:dilog}
Let $E$ be an elliptic curve over algebraically closed field of characteristic zero. For any rational function $f$ on $E$ the element $\eta_f$ can be represented as a linear combination with integer coefficient of the elements of the form $\eta_{l,m,n}$ and $[z]+[-z]$. 

This implies that when we defining the elliptic Bloch group, we can omit the elements $\eta_f$ with $\deg f>3$.
\end{theorem}

This theorem gives a solution of Conjecture 3.30 from \cite{GL}. We remark that this theorem does not follows directly from any of the statements mentioned in \cite{GL} and \cite{zagier2000classical}.

Denote by $K$ the field of rational functions on $E$. There are a lot of relations between the elements of the form $\eta_f$. More precisely, the following is true:

\begin{proposition}
\label{prop:motivation}
Let $f,g\in K\bs \{0,1\}$ and $f\ne g$. Modulo elements of the form $[z]+[-z]$, the following identity is true:
$$\eta_{f}+\eta_{g/f}+\eta_{(1-f)/(1-g)}=\eta_g+\eta_{(1-f^{-1})/(1-g^{-1})}.$$
\end{proposition}

This proposition motivates the following definition:

\begin{definition}[the pre-Bloch group] Denote by $\Z[K\bs \{0,1\}]$ the free abelian group generated by the set $K\bs\{0,1\}$. We  denote these  generators  as  $[f]$,  where $f\in K\bs\{0,1\}$. We define {\it the pre-Bloch group} of the field $K$ as the quotient of the group $\Z[K\bs \{0,1\}]$ by the following elements(so-called Abel five-term relations)
\begin{equation}
\label{five-relation}
[x]-[y]+[y/x]+\left[(1-x)/(1-y)\right]-\left[\dfrac {1-x^{-1}}{1-y^{-1}}\right],
    \end{equation}
where $x,y\in K\backslash\{0,1\}, x\ne y$.
\end{definition}

We recall that {\it degree} of a non-zero rational function $f\in K$ is the number of zeros of $f$ counting with multiplicities. (In particular, degree of any non-zero constant is equal to $0$.) Theorem \ref{thm:dilog} is an easy consequence of the following statement:

\begin{theorem}
\label{thm:main}
Let $E$ be an elliptic curve over algebraically closed field of characteristic $0$. The group $B_2(K)$ is generated by elements of the form $[f]$ with $\deg f\leq 3$.
\end{theorem}

The corresponding statement for the projective line is the main result of \cite{dupont1982generation}. It is interesting to generalise this result to curves of arbitrary genus.

\subsection{The Plan of Proofs} 
%We deduce Theorem \ref{thm:main} from the following two propositions. 
In the second section we deduce Theorem \ref{thm:dilog} from Theorem \ref{thm:main}.

Let us denote by $\mc F_n B_2(K)$ the subgroup of $B_2(K)$ generated by elements of the form $[f]$ with $\deg f\leq n$.
In the third section we will  define the notion {\it a generic} function and prove the following 

\begin{proposition}
\label{prop:sp}
Let $n\geq 3$. The group $\mc F_{n+1}B_2(K)$ is generated by the subgroup $\mc F_n B_{2}(K)$ and generic rational functions of degree $n+1$.
\end{proposition}

The main result of the fourth section is the following  

\begin{proposition}
\label{prop:ge}
Let $n\geq 3$. Let $f$ be a generic function of degree $n+1$. Then the element $[f]$ lies in the subgroup $\mc F_n B_2(K)$.
\end{proposition}

Theorem \ref{thm:main} follows immediately from these two propositions.

\section*{Acknowledgements}
The author is grateful to his supervisor Andrey Levin for setting the problem and useful discussions. 

This paper was partially supported by the Basic Research Program at the HSE University and by the Moebius Contest Foundation for Young Scientists

\section{The Deduction of Theorem \ref{thm:dilog} from Theorem \ref{thm:main}}
  Denote by $\mb Z[E]^-$ the quotient of $\mb Z[E]$ by the subgroup generated by the elements of the form $[\xi]+[-\xi]$, where $\xi\in E$. Denote by $\pi\colon \mathbb Z[E]\to\mathbb Z[E]$ the natural projection. Introduce the map $\beta\colon \Lambda^{2}K^{\times}\to \Z[E]^-$ defined by the formula $f\otimes g\mapsto \sum_{i,j}n_im_j\pi([x_i-y_j])$, where $\sum_i n_i[x_i], \sum_j m_j[y_j]$ are the divisors of the functions  $f$ and $g$. Denote by $\mc R_{as}$ the subgroup of $\mb Z[E]^-$ generated by the elements of the form $\beta(f\wedge (1-f))$ where $f\in K(E)$.   According to \cite[Lemma 1.1]{S} there is the well-defined map $\delta\colon B_2(K)\to\Lambda^2 K^{\times}$, given by the formula $\delta([f])=f\wedge (1-f)$. (It is non-trivial to prove that this map is zero on an Abel five-term relations (\ref{five-relation})). By the definition, the image of the map $\beta\circ \delta\colon B_2(K)\to \mb Z[E]^-$ coincides with $\mc R_{as}$. According to Theorem \ref{thm:main} the group $B_2(K)$ is generated by elements of the form $[f]$ with $\deg f\leq 3$.  We have proved the following statement:

\begin{corollary}
The group $\mc R_{as}$ is generated by elements of the form $\beta(f\wedge(1-f))$ with $\deg f\leq 3$. 
\end{corollary}

It is easy to see that the element $\pi(\eta_f)$  is equal to $\beta(f\wedge(1-f))$. Now Proposition \ref{prop:motivation} follows from the fact that the map $\delta$ is zero on the Abel five-term relations (\ref{five-relation}).

Let us prove Theorem \ref{thm:dilog}. According to the previous corollary the group $\mc R_{as}$ is generated by elements of the form $\beta(f\wedge(1-f))$ with $\deg f\leq 3$. It follows that the element $\eta_f$  can be represented as a sum of elements $\eta_{f'}$ with $\deg f'\leq 3$ and elements of the form $[z]+[-z]$. So it is enough to prove that the theorem holds for  functions of degree $\leq 3$. If $\deg f\leq 2$ then it is easy to see that the element $\eta_f$ is a sum of elements of the form $[z]+[-z]$. Let $\deg f=3$. Since the map $\beta$ is translation-invariant and any point on $E$ is $3$-divisible, we can assume that the sum of the zeros of $f$ is equal to zero. It is easy to see that in this case the element $\eta_f$ is equal to $\eta_{l,m,n}$ for some lines $l,m,n$.

\section{The Reduction to the Generic Case}
Let us fix some algebraically closed field $k$ of characteristic zero. Starting from here we assume that $E$ is an arbitrary elliptic curve defined over $k$.

We need the following
\begin{definition}\label{def:general} 
A rational function $f$ on elliptic curve $E$ is called a {\it generic} if there are points $\alpha_1,\alpha_2\in f^{-1}(0), \beta_1,\beta_2\in f^{-1}(1), \gamma_1,\gamma_2\in f^{-1}(\infty)$ on $E$, such that the following conditions hold
\begin{enumerate}
 \item If the points $\alpha_1,\alpha_2$ are equal then the multiplicity of zero of the function $f$ at the point $\alpha_1$ is at least $2$. In the same way if the points $\gamma_1,\gamma_2$ are equal then the multiplicity of pole of the function $f$ at the point $\gamma_1$ is at least $2$,
 \item The points $\beta_1,\beta_2$ are different from each other,
 \item The points $\alpha_1+\alpha_2,\beta_1+\beta_2$ and $\gamma_1+\gamma_2$ are mutually different,
 \item The points $\alpha_1+\alpha_2-\gamma_1-\beta_1,\alpha_1+\alpha_2-\gamma_1-\beta_2$ are non-zero.
 \end{enumerate} 
\end{definition} 

In order to prove Proposition \ref{prop:sp} we need the following

\begin{lemma}
\label{lemma:sp}
Let $n\geq 3$. Let $f$ be a function of degree $n+1$ satisfying the first two conditions of Definition \ref{def:general} and the third condition fails.  Then $[f]\in \mc F_n B_2(K)$.
\end{lemma} 

\begin{proof}Let us consider the case when $\alpha_1+\alpha_2=\beta_1+\beta_2$, the other cases are similar.
There is a function $h$ of degree $2$ on $E$ taking the value zero at the points $\alpha_1,\alpha_2$ the value $1$ at the points $\beta_1,\beta_2$ and the value infinity at the point $\gamma_1$. Let us substitute $x=h,y=f$ into the formula (\ref{five-relation}):
\begin{equation}
\label{eq:five_term2}
    [h]-[f]+[f/h]+\left[(1-h)/(1-f)\right]-\left[\dfrac {1-h^{-1}}{1-f^{-1}}\right].
\end{equation}
Let us prove that degree of the function $f/h$ is $\leq n$. The divisors of the functions $f$ and $h$ have the following forms:
$$(f)=\sum\limits_{i=1}^{n+1}([\alpha_i-\gamma_i]), (h)=[\alpha_1]+[\alpha_2]-[\gamma_1]-[\gamma'],$$

for some points $\alpha_i,\gamma_i,\gamma'\in E, i=3\dots n+1$. We have
\begin{multline}
   (f/h)=(f)-(h)=\sum\limits_{i=1}^{n+1}([\alpha_i-\gamma_i])-([\alpha_1]+[\alpha_2]-[\gamma_1]-[\gamma'])=\\=\left([\gamma']+\sum\limits_{i=3}^{n+1}[\alpha_i]\right)-\sum\limits_{i=2}^{n+1}[\gamma_i]. 
\end{multline}

So degree of the function $f/h$ is not higher than $n$. The cases of the functions $(1-f)/(1-h),(1-f^{-1})/(1-h^{-1})$ are similar. Therefore degree of all but the second term of the relation (\ref{eq:five_term2}) are $\leq n$. So $[f]\in \mc F_n B_2(K)$.

\end{proof}

The proof of Proposition \ref{prop:sp}.
Let $a\in k\backslash\{0,1\}$. If we substitute $x=a,y=f$ into the formula (\ref{five-relation}) we get the following relation in the group $B_2(K)$
\begin{equation}
\label{formula:abel_a}
[a]-[f]+[f/a]+\left[(1-a)/(1-f)\right]-\left[\dfrac {1-a^{-1}}{1-f^{-1}}\right].
\end{equation}
Let $f_{1,a}=f/a,f_{2,a}=(1-a)/(1-f),f_{3,a}=(1-a^{-1})/(1-f^{-1})$. Since the field $k$ is infinite it is enough to prove that for any $j\in\{1,2,3\}$ and all but finite values of $a$ the function $f_{j,a}$ is either a generic or the element $[f_{j,a}]$ lies in the subspace $\mc F_n B_2(K)$.

Let $\alpha_1,\alpha_2$ and $\gamma_1,\gamma_2$ be as in the first condition of Definition \ref{def:general} for the function $f$ . Let $A_0\subset \mb P^1$ be the set of critical values of the function $f$, and $A=A_0\cup\{f(\alpha_1+\alpha_2-\gamma_1)\}\cup\{0,\infty\}$. Let $a\not\in A$. The set $f^{-1}(a)$ has precisely $n+1$ elements. Let $\beta_1,\beta_2\in f^{-1}(a)$ be two different elements. It easy to see that for the function $f/a$ the first, the second and the fourth conditions of Definition \ref{def:general} hold. If the third condition is also holds then $f/a$ is a generic rational function. If it does not hold then according to Lemma \ref{lemma:sp} the element $[f/a]$ lies in the subgroup $\mc F_n B_2(K)$.

The cases of the functions $f_{2,a},f_{3,a}$ are similar. The proposition is proved.

\section{The Decreasing of Degree in the Case of  a Generic Function}
Let us denote by $\wp$ the function of degree $2$ on $E$, satisfying $\wp(z)=z^{-2}+o(z)$ at $z\to 0$. It is easy to see that for arbitrary points $\alpha, \beta\in E, \alpha\ne \beta$, the function $\wp(z-\alpha)-\wp(z-\beta)$ has the following divisor:
\begin{equation}
\label{formula_div_wp}
(\wp(z-\alpha)-\wp(z-\beta))=\left(\sum\limits_{2x\in \alpha+\beta}[x]\right)-2[\alpha]-2[\beta].
\end{equation}

We recall that the cross ratio of four points on $\mb P^1$ is defined by the following formula $$[a,b,c,d]=\dfrac {a-c}{b-c}:\dfrac{a-d}{b-d}=\dfrac {(a-c)(b-d)}{(b-c)(a-d)}.$$

Let $\alpha,\beta,\gamma,\delta$ be a four mutually different points on $E$. We define the function $h_{\alpha,\beta,\gamma,\delta}$ by the following formula:

$$h_{\alpha,\beta,\gamma,\delta}(z)=[\wp(z-\alpha),\wp(z-\beta),\wp(z-\gamma),\wp(z-\delta)].$$

It follows from the formula (\ref{formula_div_wp}) that:
\begin{equation}
\label{formula_div_h}
(h_{\alpha,\beta,\gamma,\delta})=\sum_{2x\in \{\alpha+\gamma,\beta+\delta\}}[x]-\sum_{2x\in \{\alpha+\delta,\beta+\gamma\}}[x].
\end{equation}
Since the points $\alpha,\beta,\gamma,\delta$ are mutually different degree of the function $h$ is equal to $8$. The group $E(2):=\{z\in E| 2z=0\}$ acts on $E$ by translations. Hence the divisors of the functions $h_{\alpha,\beta,\gamma,\delta}$ and $1-h_{\alpha,\beta,\gamma,\delta}=h_{\alpha,\gamma,\beta,\delta}$ are invariant under the group $E[2]$. Therefore the function $h_{\alpha,\beta,\gamma,\delta}$ is also invariant under the group $E[2]$. 
We need the following
\begin{lemma}
\label{lemma_about_h}
For any $m\in \mb P^1\bs\{0,1,\infty\}$ there is $\mu\in E$, such that the following conditions hold:
\begin{enumerate}
\item $h_{\alpha,\beta,\gamma,\delta}(\mu)=m$,
\item $\mu \not \in \{\alpha,\beta\}$,
\item $2\mu \not \in\{\alpha+\gamma,\beta+\delta, \alpha+\delta,\beta+\gamma,\alpha+\beta,\delta+\gamma\}$.
\end{enumerate}
\end{lemma} 
\begin{proof}
As a map between a two projective curves the map $h_{\alpha,\beta,\gamma,\delta}$ is surjective. Therefore the set $f^{-1}(m)$ is non-empty. 
Since the function $h_{\alpha,\beta,\gamma,\delta}$ is invariant under the group $E[2]$, the set $h_{\alpha,\beta,\gamma,\delta}^{-1}(m)$ is also invariant under the group $E[2]$. So there are at least $4$ points satisfying the first condition of the lemma. We can pick from them a point $\mu$ different from $\alpha$ and $\beta$. So the first and the second conditions for the point $\mu$ holds.  Since the point $m$ do not lie in the set $\{0,1,\infty\}$ the point $\mu$ does not lie on the divisors of the functions $h_{\alpha,\beta,\gamma,\delta},1-h_{\alpha,\beta,\gamma,\delta}=h_{\alpha,\gamma,\beta,\delta}$. Now the third statement of the lemma follows from the formula (\ref{formula_div_h}).
\end{proof}

\begin{proposition}
\label{prop:degree_2}
Let $\alpha,\beta,\gamma,\delta$ be four mutually different points on $E$ and let $a,b,c,d$ be four mutually different points on $\mb P^1$. Then there is a function $f$ of degree $2$ on $E$ such that $f(\alpha)=a, f(\beta)=b, f(\gamma)=c, f(\delta)=d$.
\end{proposition}

\begin{proof}
Let $m$ be the cross-relation of points $a,b,c,d$ and $\mu$ be the point given by Lemma \ref{lemma_about_h}. Let us define the function $f$ by the following formula
$$f(z)=\dfrac {\wp(z-\mu)-\wp(\alpha-\mu)}{\wp(z-\mu)-\wp(\beta-\mu)}.$$
The divisor of this function is equal to $(f)=[\alpha]+[2\mu-\alpha]-[\beta]-[2\mu-\beta]$.
It follows from the statement of the previous lemma that the function $f$ satisfies the following two conditions:
\begin{enumerate}
    \item The degree of $f$ is equal to $2$,
    \item $f(\gamma),f(\delta)\not \in \{0,\infty\}$.
\end{enumerate}
We have:
$$\dfrac{f(\gamma)}{f(\delta)}=\dfrac {\wp(\gamma-\mu)-\wp(\alpha-\mu)}{\wp(\gamma-\mu)-\wp(\beta-\mu)}\cdot \dfrac {\wp(\delta-\mu)-\wp(\beta-\mu)}{\wp(\delta-\mu)-\wp(\alpha-\mu)}=h_{\alpha,\beta,\gamma,\delta}(\mu)=m.$$

So the function $\tilde f$ given by the formula $\tilde f(z)=f(z)/f(\delta)$ satisfies $\tilde f(\alpha)=0, \tilde f(\gamma)=m,\tilde f(\delta)=1,\tilde f(\beta)=\infty$.
Since $[0,m,1,\infty]=[a,c,d,b]$ there is an element $g$ of the group $PSL_2(k)$ transforming the points $0,m,1,\infty$ to the points $a,c,d,b$. It is easy to see that the function $g(\tilde f(z))$ satisfies the statement of the proposition.
\end{proof}
We have the following
\begin{corollary}
\label{cor:degree_3}
Let $\alpha_1,\alpha_2,\beta_1,\beta_2,\gamma_1,\gamma_2$ be a points on $E$, such that the following conditions hold
\begin{enumerate}
\item The sets $\{\alpha_i\}_{i=1,2},\{\beta_i\}_{i=1,2},\{\gamma_i\}_{i=1,2}$ do not intersect,
\item The points $\beta_1$ and $\beta_2$ are different,
\item The points $\alpha_1+\alpha_2-\gamma_1-\beta_1,\alpha_1+\alpha_2-\gamma_1-\beta_2$ are non-zero,
\item The points $\alpha_1+\alpha_2,\beta_1+\beta_2$ and $\gamma_1+\gamma_2$ are mutually different.
\end{enumerate} 
Then there is a rational function $g$ on $E$ satisfying $g(\beta_1)=g(\beta_2)=1$ and one of following statements hold:
\begin{enumerate}
\item Degree of $g$ is equal to $3$ and its divisor is equal to $[\alpha_1]+[\alpha_2]+[\alpha']-[\gamma_1]-[\gamma_2]-[\gamma']$, where $\alpha',\gamma'$ are some points on $E$.
\item Degree of $g$ is equal to $2$ and its divisor is equal to $[\alpha_1]+[\alpha_2]-[\gamma_j]-[\gamma']$, where $j\in \{1,2\}$ and $\gamma'$ is some point on $E$.
\item Degree of $g$ is equal to $2$ and its divisor is equal $[\alpha_j]+[\alpha']-[\gamma_1]-[\gamma_2]$, where $j\in \{1,2\}$ and $\alpha'$ is some point on $E$.
\end{enumerate}
 1. is "generic" case and 2. and 3. are its degenerations.
\end{corollary}

\begin{proof}
It follows from the first condition of the corollary that there is a function $g_1$ of degree $2$  with the divisor is equal to $[\alpha_1]+[\alpha_2]-[\gamma_1]-[\alpha_1+\alpha_2-\gamma_1]$. From the first and the third conditions it follows that $g_1(\beta_1)\not \in \{0,\infty\}$. Let us denote by $\tilde f_1$ the function $g_1/g_1(\beta_1)$. From the conditions of the corollary it follows that $\tilde g_1(\beta_2)\not\in\{ 0,1,\infty\}$. According to Proposition \ref{prop:degree_2} there is a function $g_2$ of degree $2$ on $E$ taking the values $0,\infty, 1, g_1(\beta_2)^{-1}$ at the points $\alpha_1+\alpha_2-\gamma_1, \gamma_2,\beta_1,\beta_2$. Let $g:=\tilde g_1g_2$. By the construction of the function $g$ we have $g(\beta_1)=g(\beta_2)=1$.  Denote by $[\alpha_1+\alpha_2-\gamma_1]+[s]-[\gamma_2]-[t]$ the divisor of the function $g_2$ for some $s,t\in E$. We have
$$(g)=[\alpha_1]+[\alpha_2]+[s]-[\gamma_1]-[\gamma_2]-[t].$$
There are three cases:
\begin{enumerate}
    \item $s\not \in \{\gamma_1,\gamma_2\}, t\not \in \{\alpha_1,\alpha_2\}$. In this case degree of the function $g$ is equal to $3$ and the first case of the proposition holds.
    \item $s\in \{\gamma_1,\gamma_2\}$. In this case $t\not \in \{\alpha_1,\alpha_2\}$ and so the second case of the proposition holds.
    \item $t\in \{\alpha_1,\alpha_2\}$. In this case $s\not\in \{\gamma_1,\gamma_2\}$ and so the third case of the proposition holds.
    
\end{enumerate}
\end{proof}

Now we can give the proof of Proposition \ref{prop:ge}. Let $f$ be a generic rational function of degree $n+1$ on $E$. Denote by $\alpha_1,\alpha_2,\beta_1,\beta_2,\gamma_1,\gamma_2$ the corresponding points from Definition \ref{def:general}. They are satisfying the conditions of Corollary \ref{cor:degree_3}. Let $g$ be a function satisfying one of the statements of Corollary \ref{cor:degree_3}. Let us substitute $x=g,y=f$ into the relation \eqref{five-relation}
$$[g]-[f]+[f/g]+\left[(1-g)/(1-f)\right]-[(1-g^{-1})/((1-f^{-1}))].$$
Similarly to the proof of Lemma \ref{lemma:sp} it is not difficult to show that all but the second term has degree $\leq n$. So $[f]\in\mc F_n B_2(K)$.

\bibliographystyle{spmpsci}      % mathematics and physical sciences
\bibliography{mylib} 

% name your BibTeX data base
Faculty of Mathematics, National Research University Higher School of Ecnomics, Russian Federation, Usacheva str., 6, Moscow, 119048; HSE-Skoltech International Laboratory of Representation
Theory and Mathematical Physics, Usacheva str., 6, Moscow, 119048.

E-mail address: {\it vbolbachan@gmail.com}
\end{document}